\documentclass[a4paper, 12pt]{amsart}
\usepackage{amsmath,mathrsfs,amsthm}
\usepackage{txfonts,rotating,tikz}
\usepackage{esvect}

\usepackage[all]{xy}
\usepackage{nicefrac}

\usepackage[all]{xy}
\usepackage{color} 
\usepackage{hyperref}

\usepackage{nicefrac}

\newtheorem{theorem}{Theorem}

\newtheorem{definition}[theorem]{Definition}
\newtheorem{remark}[theorem]{Remark}

\newtheorem{conjecture}[theorem]{Conjecture}

\newtheorem{example}[theorem]{Example}
\def\text{\mbox}
\def\tilde{\widetilde}

  \renewcommand{\tilde}{\widetilde}

    \DeclareMathOperator{\Gr}{Gr}
         
     \DeclareMathOperator{\SL}{SL}   
           
       \DeclareMathOperator{\Sp}{Sp}  
        \DeclareMathOperator{\SO}{SO}      
  \DeclareMathOperator{\Pic}{Pic}

\newcommand{\beqn}{\begin{equation}}
\newcommand{\eeqn}{\end{equation}}

\renewcommand{\ni}{\noindent}

 \newcommand{\Pbb}{\mathbb{P}}

\newcommand{\bc}{\mathbb{C}}

\newcommand{\bp}{\mathbb{P}}
\newcommand{\br}{\mathbb{R}}
\newcommand{\bz}{\mathbb{Z}}

       \title[Nonexistence of regular maps]       
       {Nonexistence of regular maps
between homogeneous projective varieties}
           
\date{}

\author{Shrawan Kumar}

\address{S. Kumar: Department of Mathematics, University of North Carolina, Chapel Hill, NC 27599-3250, USA}
\email{shrawan@email.unc.edu}

\begin{document}

\maketitle{}
\section{Introduction}
The base field is the field of complex numbers $\bc$. We study non-existence of non-constant regular maps from a partial flag variety $X=G/P$ to another partial flag variety $X'=G'/P'$, 
where $G$ (resp. $G'$) is a  connected simple algebraic group and $P\subset G$ (resp. $P'\subset G'$) is 
a parabolic subgroup.   Observe that there is no
uniqueness of the presentation $X\cong G/P$, e.g., the projective space
$\Pbb^{2n-1}$ is the homogeneous space of $\SL(2n)$ as well as $\Sp(2n)$ (see Remark \ref{remark6} for an exhaustive list). 

 The main result of this note is the following theorem (cf. Theorem \ref{thm1}):
 \vskip1ex
 {\bf Theorem:}  {\it Let $G, G'$ be as above  and $P$ (resp. $B'$) be a parabolic subgroup of $G$ different  
from its Borel subgroup (resp. a Borel subgroup of $G'$). Then, there does not exist any non-constant regular map
from $G/P$ to  $G'/B'$.}
 \vskip1ex

 In Section 3, we recall some  results on the non-existence of non-constant regular maps from a partial flag variety $X$ to another partial flag variety $X'$ (but, by no means, an exhaustive list).  Also, we make the following conjecture (cf. Conjecture \ref{conj1}):
 \vskip1ex

{\bf Conjecture:} {\it  Let $X,X'$ be two  homogeneous projective varieties as above. Then,

  (a) Assume that  $X$ is different from $\bp^{2n}$ (for $n\geq 1$) and
\[\mbox{minss rank }X > \mbox{maxss rank }X' ,
  \]
  where $\mbox{minss rank }$ and $\mbox{maxss rank }$ are defined in Definition \ref{defi3}.
  Then, there does not exist any non-constant regular map from $X$ to $X'$.
\vskip1ex

(b) If $X =\bp^{2n}$ (for $n\geq 1$) and there exists a non-constant regular map from $X \to X'$, then 
\[\mbox{minss rank } \bp^{2n-1} = n-1 \leq  \mbox{maxss rank }X'.\]}

 {\bf Acknowledgements:} The above theorem was proved (partially supported by then NSF grant number DMS-0070679) and communicated to J. Landsberg in a letter dated September 17, 2002. We decided to publish this note now due to its renewed interest. We thank J. Landsberg and A. Parameswaran for some helpful conversations.

 \section{The main theorem and its proof}
 We prove the following theorem, which is the main result of this note. 
 \begin{theorem}  \label{thm1} Let $G, G'$ be simple and connected algebraic groups and let $P$ (resp. $B'$) be a parabolic subgroup of $G$ different  
from its Borel subgroup (resp. a Borel subgroup of $G'$). Then, there does not exist any non-constant regular map
from $G/P$ to  $G'/B'$.
 \end{theorem}

\begin{proof}  Let $H^*(G'/B')$ be the singular cohomology of the full flag variety $G'/B'$ with real coefficients. Then,  by the Bruhat decomposition of $G'/B'$,  $H^*(G'/B')$ has a Schubert basis $\{\epsilon^w\}_{w\in W'}$,  where $W'$ is the Weyl group of $G'$ and 
$\epsilon^w\in H^{2\ell(w)}(G'/B')$ ($\ell(w)$ being the length of $w$).  We declare a cohomology class  $\epsilon$ {\it non-negative} if $\epsilon= \sum_w a_w\epsilon^w$ has all its coefficients $a_w$ non-negative and write it as $\epsilon \geq 0$.   Observe that by a  classical positivity theorem (cf., e.g., [Ku1, Corollary 11.4.12]),  for any $v, w\in W'$,  the cup product of two non-negative classes  is non-negative.  

Let $f: G/P \to G'/B'$ be a regular map.  We first prove the theorem in the case $P$ is a maximal parabolic subgroup of $G$ but not a Borel subgroup (in particular, $G$ is not isogeneous to $\SL(2)$; thus  $G/P$ is not isomorphic with $\mathbb{P}^1$).  Consider the induced map $f^*: H^2(G'/B') \to H^2(G/P)$.  For any $x\in 
H^2(G'/B')$,  since $H^2(G/P)$ is one dimensional (spanned by a Schubert class denoted $\epsilon_o$),  $f^*(x)^2$ is a non-negative class.  Choose a basis $\{x_i\}$ of $H^2(G'/B') \simeq \mathfrak{h}'^*_\br$ (under the Borel isomorphism [Ku1, Definition 11.3.5]) such that $\sum_i\,x_i^2$ is $W'$-invariant (use the $W'$-invariant positive definite form on $ \mathfrak{h}^*_\br$).  Here $\mathfrak{h}'^*_\br$ is the real
span of the weight lattice of $G'$ (with respect to the fixed choice of a maximal torus in $G'$). 
Thus,  
$$\sum_i\,x_i^2= 0\in H^4(G/B),$$
since $H^*(G'/B')$ does not have any $W'$-invariant except in $H^0(G'/B')$.
Write $f^*(x_i) = d_i \epsilon_o$, for $d_i \in \br$. 
Hence,
\begin
{equation}0=f^*(\sum_i \,x_i^2)= \sum_i \left(f^*(x_i)\right)^2 = \sum_i d_i^2 \epsilon_o^2
\geq 0.
\end{equation}
Now,  $ \epsilon_o^2\in H^4(G/P)$ is non-zero by Wirtinger Theorem (cf. [GH, Page 31]). (This is here we have used the assumption that   $G/P$ is not isomorphic with $\mathbb{P}^1$.) Hence,  by the above equation,  each $d_i=0$.  Thus, $f^*(H^2(G'/B')) =0$.  This forces $f$ to be a constant since 
$$H^2(G/B) \simeq \Pic (G/B)\otimes_\bz\, \br$$
(and similarly for $G'/B'$) and a very ample line bundle on $G'/B'$ pulls-back via $f$ to a non-trivial line bundle if $f$ is non-constant.  This completes the proof of the theorem in the case $P$ is a maximal parabolic subgroup of $G$. 

We come now to the case when  $P$ is not a maximal parabolic of $G$ (and not a Borel subgroup). Take any parabolic subgroup $Q\supset P$ of $G$ such that $P$ is a maximal parabolic subgroup of $Q$ and $Q/P$ is not isomorphic with $\mathbb{P}^1$ (i.e., the unique extra simple root contained in the Levi subgroup of $Q$ is connected to one of the simple roots for the Levi subgroup of $P$ in the Dynkin diagram of $G$). Then, the variety $Q/P$ is isomorphic with a variety of the form $H/L$, where $H$ is a simple and connected algebraic group and $L$ is a maximal parabolic subgroup of $H$ different from its Borel subgroup. Thus, from the maximal parabolic case proved above, we get that $f$ restricted to $Q/P$ is constant and so is $f$ restricted to $gQ/P$ for any $g\in G$. Hence, $f$ factors through $G/Q$ as a regular map $f_Q$, i.e., we have the following commutative triangle:
\[
\xymatrix{
G/P\ar[rr]^-{\pi} \ar[rd]_{f}& & G/Q\ar[ld]^{f_Q}\\
 & G'/B' , &
}
\]
where $\pi$ is the canonical projection. Continuing this way, we get that $f$ descends to a regular map $\bar{f}: G/\bar{P} \to G'/B'$, 
where $\bar{P}$ is a maximal parabolic subgroup of $G$. Thus, the theorem follows from the case of maximal parabolic subgroups of $G$ proved above.

\end{proof}

\section{Review of some related results and a conjecture}

We recall the following results on the existence (or non-existence) of regular maps between projective homogeneous varieties:

\begin{theorem} Let $G(r,n)$ denote the Grassmannian of $r$-dimensional subspaces of $\bc^n$. 

\vskip1ex

(a) (due to Tango [T]) There does not exist any non-constant regular map from $\bp^m \to G(r, n)$ for $m\geq n$. 
\vskip1ex

(b)  (due to Paranjape-Srinivas [PS]) There exists a finite surjective regular map $f: G(r, n) \to G(s, m)$ if and only if $r=s$ and $n=m$. Moreover, in this case, $f$ is a biregular isomorphism. 

\vskip1ex

(c) (due to Hwang-Mok [HM]) Let $G$ be a simple and connected algebraic group and let $P$  be a maximal parabolic subgroup of $G$ 
and let $f: G/P\to Y$ be a surjective regular map to a
smooth projective variety $Y$ of positive dimension.   Then, either $G/P$ is
biregular isomorphic to the projective space $\Pbb^n$, $n=\dim G/P$, or $f$ is a biregular
isomorphism.
\vskip1ex

(d) (due to J. Landsberg- unpublished) There does not exist any non-constant regular map from $\bp^5 \to G(3, 6)$. Also, there is no 
non-constant regular map from $\bp^6 \to \SO(10)/P(5)$, where $P(5)$ is the maximal parabolic subgroup of $\SO(10)$ obtained from deleting the $5$-th simple root (following the convention in [Bo, Planche IV]).

\vskip1ex

(e) (due to Naldi-Occhetta [NO]) Any regular map from  $f: G(r, n) \to G(s, m)$ for $n>m$ is constant. 

\vskip1ex

 (f) (due to Mu$\tilde{n}$oz-Occhetta-Sol\`a Conde [MOS]) Let $G$ be a simple and connected algebraic group of classical type and $P$ a parabolic subgroup  different  
from its Borel subgroup. Let $M$ be a smooth complex projective variety such that $e.d. (M) > e.d. (G/P)$, where e.d. (effective good divisibility) is defined in loc cit., $\S$1. 
 Then, there are no nonconstant regular maps from $M\to G/P$. 
\vskip1ex

(g) (due to Bakshi-Parameswaran [BP]) Let $P_i$ be the minimal parabolic subgroup of $\SL(n)$ such that its Levi subgroup has (only) one  simple root $\alpha_i$ (following the convention in [Bo, Planche I]. Then, any regular map $f: \bp^3  \to \SL(n)/P_i$ is constant for $i\in \{1, n-1\}$. Moreover, any regular map $f: \bp^4  \to \SL(n)/P_i$ is constant for any $i$. Further, there exist non-constant regular maps 
 $f: \bp^3  \to \SL(n)/P_i$ for $i\in \{2, \dots, n-2\}$.  
\end{theorem} 

\begin{definition} \label{defi3} {\rm Let $X=G/P, X'=G'/P'$ be as in the beginning of Introduction.  Define the {\em minimum} (resp. {\em maximum}) {\em semisimple 
stabilizer rank} of $X$
as the minimum (resp.  maximum) of the  ranks of the semisimple part
of the Levi component of $P$ (for all possible realizations of $X$ as $G/P$, 
with $G$ a simple and
connected algebraic group and $P$ a parabolic subgroup).  Denote these ranks
by {\it minss rank} and {\it maxss rank} respectively.

Observe that the stabilizer rank of $X$ is equal to the rank of $G$ - rank of the Picard group of $X$.}
\end{definition}

 \begin{remark} \label{remark6} {\rm The list of  non-isogeneous simple and connected $G,G'$ such that $X=G/P \simeq G'/P'$ (for some parabolic subgroups $P\subset G$ and $P'\subset G'$) is as follows (cf. [D, $\S$2]\footnote{We thank M. Brion for the reference}):
 
 1. $G=\SL(2n), G'= \Sp(2n), X= \mathbb{P}^{2n-1},$ (for $n\geq 2$)
 
 2. $G=\SO(7), G'=G_2, X$ is  the quadric of dimension $5$
 
 3.  $G=\SO(2n+2), G'=\SO (2n+1)$ (for $n\geq 2$),  $X$ is the variety of isotropic subspaces of dimension $n$ in $\bc^{2n+1}$, where 
 $\bc^{2n+1}$ is equipped with a non-degenerate quadratic form. 
  }
 \end{remark}
 
We make the following conjecture:  
\begin{conjecture} \label{conj1} Let $X,X'$ be two  homogeneous projective varieties as in the above Definition. Then,

  (a) Assume that  $X$ is different from $\bp^{2n}$ (for $n\geq 1$) and
\[\mbox{minss rank }X > \mbox{maxss rank }X' .
  \]
Then, there does not exist any non-constant regular map from $X$ to $X'$.
\vskip1ex

(b) If $X =\bp^{2n}$ (for $n\geq 1$) and there exists a non-constant regular map from $X \to X'$, then 
\[\mbox{minss rank } \bp^{2n-1} = n-1 \leq  \mbox{maxss rank }X'.\]
 \end{conjecture}
 
\begin{example} (due to J. Landsberg) {\rm Consider the map 
\begin{align*} \SL(4p)/P_1=\bp^{4p-1} \to &\SO(6p)/P_1=Q^{6p-2},\,\,\,  [x_0^1, x_1^1, \dots, x_0^{2p}, x_1^{2p}] \mapsto\\
& [(x_0^1)2, (x_1^1)^2, x_0^1x_1^1, \dots, (x_0^{2p})^2, x_1^{2p})^2, x_0^{2p}x_1^{2p}],
\end{align*}
where $Q^{6p-2}$ is the smooth quadric of dimension $6p-2$. Now, restrict this map to a general hyperplane $\bp^{4p-2}$ to get a non-constant regular map from $\bp^{4p-2} \to Q^{6p-2}.$ Observe that $\mbox{minss rank } \bp^{4p-2} =4p-3$ and $\mbox{maxss rank }  Q^{6p-2}= 3p-1$.}
\end{example}
\begin{remark} {\rm Let $\{V_t\}$ be a family of rank two vector
 bundles on $\mathbb{P}^3$ parametrized by the formal disc of one dimension.  Assume that the general member of the family is
 a trivial vector bundle. Then, is the special member $V_0$ also
 a trivial vector bundle? This question is a slightly weaker version of a question by Koll\'ar and Peskine  
 on complete intersections of a family of smooth curves in $\mathbb{P}^3$. An affirmative answer of the above 
 question on vector bundles is equivalent to the non-existence of non-constant regular maps from 
  $\mathbb{P}^3\to \mathcal{X}$, where $\mathcal{X}$ is the infinite Grassmannian associated to affine $\SL(2)$
  (cf. [Ku2]).  }
    \end{remark}

\begin{center}
{\bf References}
\end{center}

\ni
[BP] \,\, Bakshi, S. and Parameswaran, A. J.: Morphisms from projective spaces to $G/P$, Preprint (2023). 

\vskip1ex

\ni
[Bo]\,\, Bourbaki, N.: Groupes et Alg\`ebres de Lie, Ch. 4-6, Masson, Paris (1981). 
\vskip1ex

\ni
[D]\,\, Demazure, M. : Automorphismes et d\'eformations des vari\'et\'es de Borel, Invent. Math. 39, 179-186 (1977). 

\vskip1ex
\ni
[GH]\,\, Griffiths, P. and Harris, J.: Principles of Algebraic Geometry, John Wiley \& Sons, Inc. (1994). 

\vskip1ex
\ni
[HM]\,\,
Hwang, J. and Mok, N.: Holomorphic maps from rational homogeneous spaces of
Picard number 1 onto projective manifolds,  Invent. Math. 136,
209--231 (1999).

\vskip1ex
\ni
[K1]\,\,
Kumar, S.:  Kac-Moody Groups, their Flag Varieties and Representation Theory, Progress in Mathematics Vol. 204, Birkh\"auser (2002). 

\vskip1ex
\ni
[K2]\,\,
Kumar, S.: An approach towards the Koll\'ar-Peskine problem via the
Instanton Moduli Space, Proceedings of Symposia in Pure Mathematics
Volume 86 (2012).

\vskip1ex
\ni
[MOS] \,\, Mu$\tilde{n}$oz, R., Occhetta, G. and Sol\'a Conde, L.E.: Maximal disjoint Schubert cycles in rational homogeneous varieties, Math. Nachrichten (2023). 

\vskip1ex
\ni
[NO]\,\, Naldi, A. and Occhetta, G.: Morphisms between Grassmannians , ArXiv:2202.11411.

\vskip1ex
\ni
[PS]\,\,
Paranjape, K.H. and Srinivas, V.: Self maps of homogeneous spaces,  
Invent. Math. 98, 425--444 (1989).

\vskip1ex
\ni
[T] Tango, H.: On $(n-1)$-dimensional projective spaces contained in the Grassmann variety $\Gr(n, 1)$, J. Math. Kyoto Univ. 14, 415--460 (1974). 
\end{document}